\documentclass[12pt,oneside,reqno]{amsart}

\usepackage{amsmath}

\usepackage{geometry}
 
\usepackage{amsthm,amssymb,mathrsfs,mathtools,dsfont}

\usepackage{graphicx}

\usepackage[colorlinks,allcolors=blue,pagebackref,hypertexnames=false]{hyperref}

\newcommand{\R}{\mathbb{R}}
\newcommand{\Z}{\mathbb{Z}}

\newcommand{\N}{\mathbb{N}}

\newcommand{\pS}{\mathbb{S}_p}
\newcommand{\ppS}{\mathbb{S}_{p}^{\geqs 0}}
\newcommand{\pppS}{\mathbb{S}_{p}^{> 0}}

\newcommand{\pnorm}[1]{\left\lVert{#1}\right\rVert_p}

\newcommand{\tetra}[2]{K_{#1}^{(#2)}}
\DeclareMathOperator{\ex}{ex}
\DeclareMathOperator{\spex}{spex}

\newcommand{\cH}{\mathcal{H}}
\newcommand{\cF}{\mathcal{F}}
\newcommand{\cP}{\mathcal{F}}

\newcommand{\fF}{\mathfrak{F}}

\newcommand{\1}{\mathds{1}}
\newcommand{\bx}{\mathbf{x}}

\newcommand{\bt}{\mathbf{t}}

\newcommand{\geqs}{\geqslant}
\newcommand{\leqs}{\leqslant}

\newtheorem{theorem}{Theorem}[section]
\newtheorem{lemma}[theorem]{Lemma}
\newtheorem{corollary}[theorem]{Corollary}
\newtheorem{conjecture}[theorem]{Conjecture}
\newtheorem{proposition}[theorem]{Proposition}
\newtheorem{question}[theorem]{Question}

\theoremstyle{definition}

\newtheorem{definition}{Definition}[section]

\theoremstyle{remark}
\newtheorem{remark}{Remark}[section]

\begin{document}
 
\title[Principal eigenvectors in hypergraph Tur\'an problems]{Principal eigenvectors and principal ratios in hypergraph Tur\'an problems}

\author{Joshua Cooper}
\address{Department of Mathematics, University of South Carolina}
\email{\href{mailto:cooper@math.sc.edu}{cooper@math.sc.edu}}

\author{Dheer Noal Desai}
\address{Department of Mathematical Sciences, The University of Memphis}
\email{\href{mailto:dndesai@memphis.edu}{dndesai@memphis.edu}}

\author{Anurag Sahay}
\address{Department of Mathematics, Purdue University}
\email{\href{mailto:anuragsahay@purdue.edu}{anuragsahay@purdue.edu}}

\dedicatory{Dedicated to the memory of Vladimir Nikiforov.}

\begin{abstract} For a general class of hypergraph Tur\'an problems with uniformity $r$, we investigate the principal eigenvector for the $p$-spectral radius (in the sense of Keevash--Lenz--Mubayi \cite{klmspex} and Nikiforov \cite{nikiforovanalyticpublished}) for the extremal graphs, showing in a strong sense that these eigenvectors have close to equal weight on each vertex (equivalently, showing that the principal ratio is close to $1$). We investigate the sharpness of our result; it is likely sharp for the Tur\'an tetrahedron problem. %Our result is heuristically sharp for the Tur\'an tetrahedron problem; it is unclear whether it is sharp for all Tur\'an problems.

In the course of this latter discussion, we establish a lower bound on the $p$-spectral radius of an arbitrary $r$-graph in terms of the degrees of the graph. This builds on earlier work of Cardoso--Trevisan \cite{cardosotrevisan}, Li--Zhou--Bu \cite{lizhoubu}, Cioab\u{a}--Gregory \cite{cioabagregory}, and Zhang \cite{zhangprincipalratio}.

The case $1 < p < r$ of our results leads to some subtleties connected to Nikiforov's notion of $k$-tightness, arising from the Perron-Frobenius theory for the $p$-spectral radius. We raise a conjecture about these issues, and provide some preliminary evidence for our conjecture.

\end{abstract}

\maketitle

\section{Introduction}

Throughout this article, we work with simple, $r$-uniform hypergraphs and we sometimes restrict to the case $r = 3$. If $V = V(G)$ and $E = E(G)$ are the vertex and edge sets respectively of an $r$-graph $G$, then, $E \subseteq \binom{V}{r}$ and, for $\bx \in \R^V$ and $e \subseteq V$ (which will often but not always be an edge), we write
\[ x^e = \prod_{v \in e} x_v, \]
and define the Lagrangian polynomial $L_G(\bx)$ by
\[ L_G(\bx) = r! \sum_{e \in E} x^e. \]

A central problem in extremal combinatorics is the Tur\'an  $(r,t)$-problem \cite{turanconj}, $t > r$ which asks for the maximum number of edges in an $r$-graph $G$ avoiding a $\tetra{t}{r}$ (i.e., the complete $r$-graph on $t$ vertices). While Tur\'an himself solved the graph case, $t>r = 2$ \cite{turanthm}, which by now has several different proofs, the problem is open for any $t>r>2$.

More generally, this problem fits into the framework of forbidden subgraph problems in extremal hypergraph theory. Let $\cH$ be a set of $r$-uniform hypergraphs. We define the graph family $\fF(\cH)$ as being $\cH$-free -- that is, $G \in \fF(\cH)$ if no graph from $\cH$ appears in $G$ as a (not necessarily induced) subgraph. The Tur\'an number $\ex_r(n,\cH)$ is defined by
\[ \ex_r(n,\cH) := \max\{ |E(G)| : G \in \fF(\cH), |V(G)| \leqs n\}, \]
where $G$ here is an $r$-uniform graph. The Tur\'an $(r,t)$-problem is then determining $\ex_r(n,\tetra{t}{r})$, where we omit the braces for convenience since $\cH = \{\tetra{t}{r}\}$ is a singleton. We refer the reader to Keevash's comprehensive survey \cite{keevashsurv} for a detailed account of hypergraph Tur\'an-type problems.

A fruitful class of extremal problems in spectral graph theory arise from looking at spectral analogues of classical Tur\'an-type problems. For an ordinary graph $G$, the spectral radius (of its adjacency matrix) $\rho_2(G)$ is a good substitute for the number of edges. This leads to the so-called ``spex" or spectral (Tur\' an) problem of determining
\begin{equation}\label{eqn: spex22} \spex_{2,2}(n,\cH) := \sup\{ \rho_2(G) : G\in \fF(\cH), |V(G)| \leqs n\}, \end{equation}
where in this definition $G$ is an ordinary graph. Although some isolated results were known already, the systematic study of spectral Tur\' an problems was initiated by Nikiforov \cite{Nikiforovpaths}. Furthermore, while the study of $\spex_{2,2}$ is interesting in its own right, it also has applications to studying $\ex_2$; for example, Nikiforov \cite{Nikiforov07} proved that the unique spectral extremal graph for $K_{s+1}$ is the Tur\' an graph $T_s(n)$. On combining this result with the straightforward observation that the spectral radius upper bounds the average degree of the graph, Nikiforov's result implies the upper bound on $\ex_2(K_{s+1})$ implicit in Tur\'an's theorem. In addition, Nikiforov \cite{NZ}, Babai and Guiduli \cite{BG} also proved spectral versions of the K\H{o}vari-S\' os-Tur\' an theorem for complete bipartite graphs $K_{s, t}$. Moreover, combining these with the observation on average degrees produces upper bounds that match the best known upper bounds for $\ex_2(n, K_{s, t})$, obtained by F\" uredi \cite{F2}.  
Very recently, several more papers determining $\spex_{2,2}(n, F)$ for various families of graphs have been published (see \cite{cioabua2022spectral, cioabua2022spectralerdossos,  LLT, NO, SSbook, Wilf86, YWZ, ZW, ZWF}) and this line of research has gained renewed interest.
In fact, spectral Tur\' an problems fit into a broader framework of problems called Brualdi-Solheid problems \cite{BS} that investigate the maximum spectral radius over all graphs belonging to a given family of graphs. Numerous results may be interpreted as belonging to this framework (see \cite{BZ, BLL, EZ, FN, nosal1970eigenvalues, S, SAH}).

This raises naturally the question of spectral analogues of hypergraph Tur\'an problems. Several problems of this nature have been considered (see \cite{ElLuWa22,GaChHo22,HoChCo21,klmspex} for a selection of results in this direction), but there is no comprehensive theory for these types of problems. One major roadblock in this avenue is the fact that there are several not necessarily equivalent ways in the literature to define the spectrum and spectral radius of a hypergraph. We adopt the following perspective on spectral radius due to Keevash--Lenz--Mubayi \cite{klmspex}, explored further by Nikiforov \cite{nikiforovanalyticpublished}:
\begin{definition}[Keevash--Lenz--Mubayi]
For any $1\leqs p \leqs \infty$, we denote the \emph{$p$-spectral radius} of $G$ by $\rho_p(G)$, which is given by
\[ \rho_p(G) = \sup_{\bx \in \pS} L_G(\bx) = \sup_{\pnorm{\bx} = 1} L_G(\bx), \]
where $\pS = \{ \bx \in \R^V : \pnorm{\bx} = 1\}$ is the unit ball in the $\ell^p$ norm.
\end{definition}
\begin{remark} \label{rem: rhoinlit}
The quantity $\rho_1(G)$ is sometimes called the \emph{Lagrangian} in the literature. The quantity $\rho_r(G)$ is the maximum eigenvalue with respect to the algebraic definition of hypergraph spectra provided by Cooper and Dutle \cite{cooperdutle}. The quantity $\rho_2(G)$ is the notion of hypergraph spectral radius introduced by Friedman and Wigderson \cite{friedmanwigderson}. Finally, $\rho_\infty(G) = r! |E(G)|$.
\end{remark}
\begin{remark} In analogy with the $p = r = 2$ case, we shall call a vector $\bx \in \pS$ which maximizes $L_G(\bx)$ a \emph{principal eigenvector} for the $\ell^p$ norm. Note that there is always a principal eigenvector with $\bx \geqs 0$, and hence the supremum in the above definition could be taken over $\ppS = \{ \bx \in \pS : x_v \geqs 0 \text{ for all } v \in V\}$ instead. For the rest of this note, when we say principal eigenvector, we always assume that it is component-wise nonnegative. 
\end{remark}

\begin{remark} \label{rem: perronfrobeniusfails} The Perron-Frobenius theory for $p = r = 2$ might suggest that the principal eigenvector $\bx$ of a connected $r$-graph $G$ is unique and satisfies $x_v > 0$ for every $v \in G$. Nikiforov \cite[Theorem~5.10]{nikiforovanalyticpublished} showed that this is true if $p \geqs r \geqs 2$. For $1 < p < r$ on the other hand, by \cite[Proposition~4.8]{nikiforovanalyticpublished} principal eigenvectors may not be unique and by \cite[Proposition~4.3]{nikiforovanalyticpublished} some principal eigenvectors might have $x_v = 0$ for some $v \in G$.  \end{remark}

\begin{remark}
In light of the previous remark, we will say `a' principal eigenvector instead of `the', unless uniqueness is clear. Further, we define
\[ \pppS = \{ \bx \in \pS : x_v > 0 \text{ for all } v \in V\}, \]
and if $\bx \in \pppS$ is a principal eigenvector of $G$, then we call it a \emph{Perron-Frobenius eigenvector}. 
\end{remark}

From this perspective, the analogue of \eqref{eqn: spex22} is given by
\begin{equation}\label{eqn: spexrp} \spex_{r,p}(n,\cH) := \sup\{ \rho_p(G) : G\in \fF(\cH), |V(G)| \leqs n\}, \end{equation}
where here $G$ is now an $r$-graph. The Tur\'an density is defined by
\[ \pi_r(\fF(\cH)) := \lim_{n\to\infty}\binom{n}{r}^{-1} \ex_r(n,\cH) = \lim_{n\to\infty} r! n^{-r} \ex_r(n,\cH), \]
which exists due to \cite{KaNeSi64}. This amounts to studying the highest order term of $\ex_r(n,\cH)$ except in degenerate cases (e.g., if $r = 2$ and every graph in $\cH$ is bipartite), where it is $0$. One can analogously define the spectral Tur\'an density by
\[ \lambda_{r,p}(\fF(\cH)) := \lim_{n\to\infty}n^{-r(1-1/p)} \spex_{r,p}(n,\cH). \]

A result of Nikiforov \cite[Theorem 9.3]{nikiforovanalyticpublished} implies the following:
\begin{theorem}[Nikiforov] \label{ethm: nikiforov}
For $\cH$ a set of $r$-graphs and $p > 1$, 
\[ \lambda_{r,p}(\fF(\cH)) = \pi_r(\fF(\cH)). \]
\end{theorem}
\begin{remark} One may view this theorem as saying that the combinatorial and spectral Tur\'an problems are asymptotically the same.\end{remark}

We now return to our discussion of the Tur\'an $(r,t)$-problem. The smallest unsolved case of the $(r,t)$-problem is $r = 3$, $t = 4$, which is sometimes called the Tur\'an tetrahedron problem, and will be of special interest to us. Tur\'an showed that $\pi_3(\tetra{4}{3}) \geqs \frac{5}{9}$ by constructing a family of $\tetra{4}{3}$-free graphs and conjectured that this bound is optimal. Brown \cite{brownconst} later showed that on $n = 3k$ vertices, there are $\asymp n$ different graphs which achieve the conjectural optimum. Further, Kostochka \cite{kostochkaconst} found classes of constructions subsuming both Tur\'an's and Brown's, thereby showing that there are exponentially many graphs on $n = 3k$ vertices achieving the conjectural optimum. Finally, Frohmader \cite{frohmaderconst} modified Kostochka's construction to include the cases $n = 3k+1$ and $n = 3k+2$. This glut of conjecturally optimal examples perhaps explains some of the difficulties in approaching this problem.

At this point it is worth remarking that all the constructions given above are extremely technical. However, Fon-der-Flaas \cite{fdfconst} provided a method for constructing a $3$-graph $G(\Gamma)$ from a simple directed $2$-graph $\Gamma$ (i.e., no loops or multiple edges) and then showed that $G(\Gamma)$ avoids the complement of $\tetra{4}{3}$ if and only if $\Gamma$ avoids a directed $4$-cycle. This gives a nice conceptual way to think about Kostochka's construction. 

Recall the standard notation
\[ \Delta(G) := \max_{v \in V(G)} \deg v, \qquad \delta(G) := \min_{v \in V(G)} \deg v, \]
for the maximum and minimum degree of $G$. Our first main result is motivated by the following observation (whose proof we leave to the reader), which appears to have gone unremarked in the literature:
\begin{proposition} \label{prop: kostochkaregular}
All graphs in Kostochka's construction are regular. Furthermore, all graphs $G$ in Frohmader's modified construction are almost regular, and satisfy
\[ \Delta(G) = \delta(G) + \Theta(n),\]
as $n = |V(G)| \to \infty$.
\end{proposition}
\begin{remark} One has that $\Delta(G), \delta(G) \asymp n^2$. Further, a standard cloning argument shows that $\Delta(G) - \delta(G) < n$ if $G$ is an extremal graph on $n$ vertices which is $\cH$-free if $\cH$ is $2$-covering (i.e., if for every $H \in \cH$, $u,v\in V(H)$, there is an edge $e \supset \{u,v\}$). The above proposition implies that the upper bound is essentially sharp for Frohmader's modified construction. \end{remark}
If an $r$-graph $G$ is regular, then it has a principal eigenvector $\bx \in \ppS$ which is uniform (i.e. $x_v = n^{-1/p}$ for all $v \in V$). One standard measure of non-uniformity is the principal ratio.

\begin{definition} For $1< p<\infty$, an $r$-graph $G$, and a principal eigenvector $\bx \in \ppS$ of $G$, we define the \emph{principal ratio} of $(G,\bx)$ by
\[ \gamma(G,\bx) = \sup_{u,v \in V(G)} \frac{x_u}{x_v}. \]
\end{definition}
\begin{remark}
When the principal eigenvector is unique, or when the dependence on the choice of principal eigenvector cannot cause confusion, we will suppress $\bx$ from the notation.
\end{remark}
\begin{remark}
If a principal eigenvector satisfies $\bx \in \ppS \setminus \pppS$, which may happen as noted in Remark~\ref{rem: perronfrobeniusfails}, then we set $\gamma(G,\bx) = \infty$. 
\end{remark}

It is clear that $\gamma(G,\bx) = 1$ if $G$ has a uniform principal eigenvector $\bx$, and further that $\gamma(G,\bx)\geqs 1$ in general. One may then ask, motivated by Proposition~\ref{prop: kostochkaregular} and the discussion thereafter, whether a principal eigenvector of an extremal graph for a Tur\'an problem is always almost uniform. Our first main result says that for several classes of spectral Tur\'an problems this is essentially true. To state this, we introduce some more notation,
\[ \cF(n) := \{ G \in \cF : |V(G)| \leqs n \},\]
\[ \Lambda_p(\cF,n) := \sup_{G \in \cF(n)} \rho_p(G), \]
where $\cF$ is any family of graphs, and we write $\fF(\cH,n)$ for $\cF(n)$ when $\cF = \fF(\cH)$.

Our theorem is then,
\begin{theorem}
\label{thm: principalupper}
Let $\cH$ be a set of $2$-covering $r$-graphs such that no graph in $\cH$ is $r$-partite, let $n \in \N$, and let $1<p<\infty$. Further, suppose that $G\in \fF(\cH,n)$ is a graph such that $\rho_p(G) = \Lambda_p(\fF(\cH),n)$. Finally, let $\bx \in \pppS$ be a Perron-Frobenius eigenvector of $G$. Then, 
\[ \gamma(G,\bx) = 1 + O_{p,r}(n^{-1}).\] 
\end{theorem}
%\textcolor{blue}{Note: The above theorem statement does not require $G$ to be a spectral extremal graph, but only that it has asymptotically the same spectral radius (i.e. same spectral Tur\' an density) as the spectral extremal graph. Further, we should confirm if Nikiforov's result, mentioned below, is for spectral extremal graphs or for any feasible graph with matching spectral Tur\' an denisity.}
\begin{remark}
%%% Strengthen the language "more restrictive hypothesis"
This sharpens an inequality essentially due to Nikiforov \cite[pp. 24-26]{nikiforovanalyticpreprint} albeit under more restrictive hypotheses; the weaker bound
\[ \gamma(G,\bx) = 1 + O_{p,r}\bigg(\frac{1}{\log n}\bigg) \]
follows simply from his Claim B, and holds for all extremal graphs $G\in \cF(n)$ in any hereditary family $\cF$. 
\end{remark}

\begin{remark} \label{rem: erdosdensity} It follows from a classical result of Erd\H{o}s \cite{erdoshypergraphESS} that the hypothesis that no graph in $\cH$ is $r$-partite is equivalent to the assertion that $\fF(\cH)$ is dense, that is to say that $\pi_r(\fF(\cH)) > 0$. It is the density of $\fF(\cH)$ that is necessary for the argument to work. \end{remark}
\begin{remark}
Theorem~\ref{thm: principalupper} will follow from a general result about spectral-extremal graphs in so-called clonal families. See Section~\ref{sec: families} and Proposition~\ref{prop: principalupper}.\end{remark}

One may ask whether Theorem~\ref{thm: principalupper} is sharp. To facilitate our investigations of this question, we have our second main result, which is of independent interest:

\begin{theorem} \label{thm: CTforpnorm}
Let $G$ be a connected $r$-graph and let $1< p < \infty$. Then, for any nonnegative principal eigenvector $\bx \in \ppS$ of $G$, 

\[ \gamma(G,\bx) \geqs \bigg(\frac{\Delta(G)}{\delta(G)}\bigg)^{\frac{1}{p+r-2}}. \] 
\end{theorem}
\begin{remark}
This generalizes a result of Cardoso--Trevisan \cite[Corollary 3.3]{cardosotrevisan} and Li--Zhou--Bu \cite[Theorem 3.1(i)]{lizhoubu} who proved the case $p = r$, which itself built on work of Cioab\u{a}--Gregory \cite[p. 373]{cioabagregory} and Zhang \cite[Theorem 2.3]{zhangprincipalratio} who independently proved the case $p = r = 2$. Our proof proceeds along similar lines. 
\end{remark}

To see what this yields about the sharpness of Theorem~\ref{thm: principalupper}, we introduce a new definition.
\begin{definition} We say a set of $r$-graphs $\cH$ is \emph{coarse}, if no graph in $\cH$ is $r$-partite, and furthermore, if for every $G \in \fF(\cH,n)$ such that $\rho_p(G) = \Lambda_p(\fF(\cH),n)$ one either has that $\Delta(G) = \delta(G)$ or that 
\[ \Delta(G) - \delta(G) =\Omega(n^{r-2}), \]
where the implicit constant here depends only on $\cH$. 
\end{definition}

\begin{remark} \label{rem: r=2coarse}Note that for $2$-graphs, this is trivially not asking more of $\cH$ than the property that it has no bipartite graph as a member, since if $\Delta \neq \delta$, then $\Delta - \delta \geqs 1$. 
\end{remark}

\begin{remark} Given Theorem~\ref{ethm: nikiforov}, it is reasonable to suspect that for many $\cH$ (particularly for those with $\fF(\cH)$ dense), if $G$ is a spectral-extremal graph in $\fF(\cH)$, then there must be an edge-extremal graph $G' \in \fF(\cH)$ such that $G$ and $G'$ are close in structure. If this is the case for the Tur\'an tetrahedron problem, then Proposition~\ref{prop: kostochkaregular} suggests that perhaps $\{ \tetra{4}{3}\}$ is coarse. This is the observation that motivated Theorem~\ref{thm: CTforpnorm} above and Corollary~\ref{corr: sharpness} below. \end{remark}

\begin{remark} While, unsatisfactorily, we are unable to actually show that $\{\tetra{4}{3}\}$ is coarse, there are examples in the literature which are provably coarse; see Remark~\ref{rem: coarseexamples} below.
\end{remark} 

\begin{corollary} \label{corr: sharpness}
Let $G$ and $\bx$ be as in Theorem~\ref{thm: principalupper}, with the additional assumption that $\cH$ is coarse and $G$ is not regular. Then, one has
\[ \gamma(G,\bx) = 1 + \Theta_{p,r}(n^{-1}).\] 
\end{corollary}

\begin{proof} The upper bound implicit in the assertion follows from Theorem~\ref{thm: principalupper}. For the lower bound, note that since $\cH$ is coarse and $G$ is not regular, it follows that
\[ \frac{\Delta(G)}{\delta(G)} - 1 = \frac{\Delta(G)- \delta(G)}{\delta (G)} \gg_r \frac{n^{r-2}}{n^{r-1}}  = n^{-1}\]
whence the lower bound follows by applying Theorem~\ref{thm: CTforpnorm} to $G$, and using the estimate $(1+x)^h = 1 + \Omega_h(x)$ as $x \to 0^+$. 
\end{proof}

\begin{remark} \label{rem: coarseexamples} There are several examples in the literature where the spectral-extremal graphs $G$ in $\fF(\cH)$ are well-understood, at least in the limit $V(G) \to \infty$. Since coarseness is an asymptotic property, this permits one to routinely check whether the sets of forbidden graphs involved are coarse. In all examples of this kind in the literature that the authors are aware of, the set of forbidden graphs is coarse. This includes the Fano plane \cite[Corollary 1.6]{klmspex} and color-critical $r$-graphs with chromatic number $r+1$ \cite[Corollary 1.5]{klmspex}. Putting the focus on the family $\fF(\cH)$ instead of $\cH$, one also has the examples of cancellative $3$-graphs \cite{ni2022spectral} and $\ell$-partite $r$-graphs with $\ell > r$ \cite{kpartitespex}. %See also Remark~\ref{rem: clonallitexample}.
\end{remark}

This raises the following question:
\begin{question} Is there a set $\cH$ such that $\fF(\cH)$ is dense, but $\cH$ is not coarse?
\end{question}
We do not address this question, and would be very interested in a resolution of it. By Remark~\ref{rem: r=2coarse}, any such example must have $r \geqs 3$. 

%%% "even though we've been assuming H is 2-covering, the general case is interesting..."

Theorem~\ref{thm: principalupper} and Corollary~\ref{corr: sharpness} also raise the interesting question of when a principal eigenvector of an extremal graph is Perron-Frobenius. Nikiforov \cite[Section 5]{nikiforovanalyticpublished} extensively studied the general Perron-Frobenius theory in this setting. As discussed in Remark~\ref{rem: perronfrobeniusfails}, if $p \geqs r$, then all principal eigenvectors are also Perron-Frobenius. To discuss what happens in the regime $1 < p < r$, we recall the following definition:
\begin{definition}[Nikiforov] Let $G$ be an $r$-graph with $E(G) \neq \emptyset$ and $1 \leqs k \leqs r-1$. We say that $G$ is \emph{$k$-tight} if, for any $U \subsetneq V(G)$ such that $U$ induces an edge, there is an edge $e \in E(G)$ such that
\[ k\leqs |e \cap U| \leqs r - 1. \]
\end{definition}
\begin{remark}
It is easy to see that $1$-tightness is equivalent to connectedness for nonempty graphs. 
\end{remark}
Nikiforov \cite{nikiforovanalyticpublished} showed that $k$-tightness was a sufficient condition for the positivity part of Perron-Frobenius theory to work. 
\begin{theorem}[Nikiforov]
Let $G$ be an $r$-graph, $1 \leqs k \leqs r-1$, and $p > r-k$. Then, if $G$ is $k$-tight and $\bx$ is a principal eigenvector of $G$, then $\bx \in \pppS$. That is,
\[ x_v > 0 \text{ for all } v \in V(G). \]
\end{theorem}

In view of the above discussion, we ask the following question:
\begin{question}
Let $\cH$ be set of $r$-graphs, and suppose that $G \in \fF(\cH)$ is spectral-extremal (that is, $\rho_p(G) = \Lambda_p(\fF(\cH),n)$ where $n = |V(G)|$). Under what conditions on $\cH$ does it follow that $G$ must be $k$-tight?
\end{question}

This question appears hard. Recall that $G \in \fF(\cH)$ is called edge-maximal if for any $G'$ on the same vertex set, $G' \in \fF(\cH)$ implies $G' \subseteq G$. It is a standard fact that spectral-extremal graphs in $\fF(\cH)$ must be edge-maximal. Thus, a simpler question is:
\begin{question} \label{ques: maximaltight}
Let $\cH$ be set of $r$-graphs, and suppose that $G \in \fF(\cH)$ is edge-maximal. Under what conditions on $\cH$ does it follow that $G$ must be $k$-tight?
\end{question}

We can partially answer this question. To state our results in this regard, we need some more definitions.

\begin{definition}
If $G$ is an $r$-graph and $1 \leqs k \leqs r-1$, then a \emph{$k$-bridge} in $G$ is an edge $e \in E(G)$ for which there exists a partition $V(G) = A \sqcup B$ so that $e$ is the unique edge with the property that 
\[ |e \cap A| \geqs k \text{ and } e \cap B \neq \emptyset. \]
An $r$-graph is called \emph{$k$-bridged} if it has a $k$-bridge and \emph{$k$-bridgeless} if it does not.
%Finally, a connected, $k$-bridged $r$-graph is called \emph{strictly $k$-bridged} if deleting a $k$-bridge disconnects the graph, and the corresponding edge is called a \emph{strict $k$-bridge}. 
\end{definition}

%\begin{remark} \label{rem: 1bridge} When $r = 2$, then $k = 1$, and the notion of a $1$-bridge is the usual notion of a bridge in graph theory. In particular, note that a $1$-bridge is also an edge cut, and furthermore, it is a strict $k$-bridge for every $1 \leqs k \leqs r-1$, possibly after switching the roles of $A$ and $B$. \end{remark}

\begin{remark} \label{rem: 1bridge} When $r = 2$, then $k = 1$, and the notion of a $1$-bridge is the usual notion of a bridge in graph theory. Note also that a $1$-bridge for $r \geqs 3$ is also an edge cut. \end{remark}

\begin{definition} \label{def: k-tightfamily}
A family of graphs $\cF$ is called a \emph{$k$-tight family} if for every edge-maximal $G \in \cF$ we have that $G$ is a $k$-tight graph. \end{definition}

Our third main result is the following:
\begin{theorem} \label{thm: bridgelessimpliestight}
Let $H$ be a $k$-bridgeless $r$-graph for $1 \leqs k \leqs r-1$. Then $\fF(H)$ is a $k$-tight family.
\end{theorem}
\begin{remark} Another way to state this theorem is that if $H$ is $k$-bridgeless and $G$ is edge-maximal with respect to the property of being $H$-free, then $G$ is $k$-tight. In particular, any principal eigenvector of $G$ for $p > r - k$ is a Perron-Frobenius eigenvector.
\end{remark}
We believe that Theorem~\ref{thm: bridgelessimpliestight} is sharp; if true, this would completely resolve Question~\ref{ques: maximaltight}. That is, we conjecture the following:
\begin{conjecture} \label{conj: bridgelessifftight}
Let $H$ be a connected $r$-graph and $1\leqs k \leqs r-1$. Then
\[ H \text{ is }k\text{-bridgeless} \iff \fF(H) \text{ is a }k\text{-tight family}. \]
\end{conjecture}

As evidence for the above conjecture we give a construction (Theorem~\ref{thm2: plateau} in Section~\ref{sec: tightfamilies}) that verifies the above conjecture under a technical assumption on $H$ (see Definition~\ref{def: plateaud}). We make little effort here to optimize this construction and we hope to return to the $r = 3$ case of Question~\ref{ques: maximaltight} in a later paper. For now, we content ourselves with highlighting the fact that the arguments behind Theorem~\ref{thm2: plateau} will settle the simplest case of Conjecture~\ref{conj: bridgelessifftight}. 

\begin{corollary} \label{corr: 1bridgelessiff1tight}
Let $H$ be a connected $2$-graph. Then
\[ H \text{ is }1\text{-bridgeless} \iff \fF(H) \text{ is }1\text{-tight}. \]
\end{corollary} 

\subsection*{Acknowledgements} This project began at the AIM workshop ``Spectral graph and hypergraph theory: connections and applications''. We thank the American Institute of Mathematics, as well as the organizers of the workshop, Sebastian Cioab\u{a}, Krystal Guo, and Nikhil Srivastava for facilitating our collaboration. We also thank Himanshu Gupta,  Leslie Hogben, Xizhi Liu, and Dhruv Mubayi for useful discussions. AS is partially supported through Purdue University start-up funding available to Trevor Wooley.

\subsection*{Organization of the paper} In Section~\ref{sec: notation}, we collect, for the convenience of the reader, a list of common notations, and some definitions. In Section~\ref{sec: preliminaries}, we state some useful lemmata, including some preliminaries from the spectral theory of hypergraphs. In Section~\ref{sec: families}, we introduce the more general setting of clonal families and state our results in that context. In Section~\ref{sec: principalupper}, we prove Theorem~\ref{thm: principalupper} via Proposition~\ref{prop: principalupper}. In Section~\ref{sec: CTforpnorm}, we prove Theorem~\ref{thm: CTforpnorm}. Finally, in Section~\ref{sec: tightfamilies}, we prove Theorem~\ref{thm: bridgelessimpliestight} and report our partial progress on Question~\ref{ques: maximaltight} and Conjecture~\ref{conj: bridgelessifftight} by proving Theorem~\ref{thm2: plateau} and its corollaries.

\section{Notation} \label{sec: notation}

We use $A = B + O(C)$ equivalently with $A - B \ll C$ only when $C$ is positive to mean that there exists a constant $K > 0$ such that $|A - B| \leqs KC$, where $A, B, C$ are quantities which depend on various parameters, but the constant $K$ shall be uniform in those parameters. If $f \ll g$ and $g \ll f$, we write $f \asymp g$ or $f = \Theta(g)$. In some cases, the quantity $K$ may be functions of other parameters in the hypotheses, so for clarity, we sometimes subscript these asymptotic notations by the object or quantity they depend upon.

Recall that, for a family $\cF$, $\cF(n)$ is the family of all graphs in $\cF$ with at most $n$ vertices. Further recall,
\[ \Lambda_p(\cF,n) = \sup_{G \in \cF(n)} \rho_p(G). \]
We also define 
\[ \Pi(\cF,n) = \frac{1}{r!} \Lambda_\infty(\cF,n),\]
so that as per Remark~\ref{rem: rhoinlit}, $\Pi(\cF,n)$ is the maximum number of edges in any graph in $\cF$ with $n$ vertices. When there is no possibility of confusion, we may suppress $p$, $\cF$, or $n$ from the notation. 

%% [Other things?]

\section{Preliminaries and lemmata} \label{sec: preliminaries}

\subsection{Spectral theory of hypergraphs}
The main tool we require from the spectral theory is a fundamental lemma in the form of the eigenequation.
\begin{lemma}[Principal Eigenequation; {\cite[Theorem~3.1]{nikiforovanalyticpublished}}] \label{lem: eigen} Let $r \geqs 2$, $1 < p < \infty$, and $G$ be an $r$-graph. Then, any nonnegative principal eigenvector $\bx \in \ppS$ satisfies, for every $v \in V(G)$,

\[ \rho_p(G) x_v^{p-1} = \frac{1}{r} \frac{\partial}{\partial x_v} L_G(\bx) = (r-1)!\sum_{\substack{e \in E(G)\\ v \in e}} x^{e - v}. \]
\end{lemma}

\subsection{Dense families and properties thereof}
We introduce now a notion of dense families; our results only apply to these families.
\begin{definition} A family $\cF$ is called \emph{dense} if 
\[ \Pi(\cF,n) \gg_\cF n^r. \]
as $n \to \infty$.
\end{definition}

\begin{lemma} \label{lem: density} If $\cF$ is a dense family, then for all $1 \leqs p < \infty$, one has
\[ \Lambda_p(\cF,n) \asymp_\cF n^{r(1-1/p)}. \]
\end{lemma}
\begin{proof}
The upper bound holds for all $\cF$, not necessarily dense by applying \cite[Proposition~2.6]{nikiforovanalyticpublished} to get that
\[ \rho_p(G) \leqs n^{r(1-1/p)},\]
for every $G$, and then taking supremum over $G\in \cF(n)$. 

For the lower bound, it suffices to find $G \in \cF(n)$ and $\bx \in \ppS$ so that
\[ L_G(\bx) \gg_\cF n^{r(1-1/p)}. \]
We let $G$ be the graph so that $|E(G)| \gg_\cF n^r$ which exists due to density of $\cF$, and set $\bx = n^{-1/p} \1$. Then,
\[ L_G(\bx) = \sum_{e \in E(G)} (n^{-1/p})^r = n^{-r/p} |E(G)| \gg_\cF n^{r(1-1/p)},\]
as desired. 
\end{proof}

\subsection{Glossary for integer partitions} While not strictly necessary, the language of integer partitions will be convenient to state Theorem~\ref{thm2: plateau}. Recall that one says $\lambda$ \emph{partitions} $r$, and one writes $\lambda \vdash r$ when $\lambda = (\lambda_1,\cdots,\lambda_\ell)$ is a tuple of positive integers satisfying 
\[ \lambda_1 \geqs \lambda_2 \geqs \cdots \geqs \lambda_\ell\] and 
\[ r = \lambda_1 + \lambda_2 + \cdots + \lambda_\ell. \]
The integer $\ell = \ell(\lambda)$ is called the \emph{length} of the partition, and the integers $\lambda_j$ are called the \emph{parts}. We call $\lambda_1$ the \emph{largest part} and $\lambda_\ell$ the \emph{smallest part}. If $\lambda,\mu \vdash r$, we say that $\mu$ \emph{refines} $\lambda$ if the parts of $\lambda$ can be subdivided to produce the parts of $\mu$. Finally, we call the partition $(r) \vdash r$ the trivial partition, and call all other partitions nontrivial.
\section{Clonal families and related notions} \label{sec: families}

The technique of cloning or Zykov symmetrization \cite{Zy49} will be frequently used in our arguments. For a graph $G$, and vertices $u,v \in V(G)$, we define $G_{u\to v}$ to be the graph obtained by deleting all edges incident to $u$, and then, for every $e \in E(G)$ such that $v \in e$ and $u \notin e$, adding the edge $e + u - v$. This is the technique of Zykov symmetrization, and we say that we have cloned $v$ onto $u$. This leads us to a central notion, that characterizes the families of graphs to which our results apply. 

\begin{definition}
A family of graphs $\cF$ is called \emph{clonal} if for every $G \in \cF$ and every $u,v \in G$, $G_{u\to v} \in \cF$.
\end{definition}

In other words, a clonal family is closed under Zykov symmetrization. The following proposition is straightforward and appears elsewhere, but we include the proof for completeness. 
\begin{proposition}\label{prop: 2-cover give clonal} Let $\cH$ be any set of $2$-covering graphs. Then the family $\fF(\cH)$ is clonal. \end{proposition}
\begin{proof} We need to show that if $G$ is $\cH$-free, then for every $u,v \in G$, $G_{u\to v}$ is $\cH$-free. To see this, suppose that $G_{u\to v}$ contains a copy $H$ of a graph in $\cH$. If $H$ does not contain $u$, then every edge in $H$ is already in $G$, contradicting the fact that $G$ cannot contain $H$. Thus, $u \in V(H)$. On the other hand, by construction, any edge involving $u$ in $G_{u\to v}$ arises from an edge involving $v$ but not containing $u$ in $G$. 
Moreover, $H$ is a $2$-covering graph but $u$ and $v$ do not share any edges in $G_{u \to v}$, so the copy of $H$ in $G_{u \to v}$ does not contain $v$.
From this we find that there must be a copy $H'$ of $H$ in $G$, which is formed by replacing $u$ by $v$. This also contradicts the fact that $G$ is $\cH$-free, completing the proof.
\end{proof}

In Section~\ref{sec: principalupper}, we shall prove the following proposition.
Theorem~\ref{thm: principalupper} will then follow:

\begin{proposition} \label{prop: principalupper}
Let $\cF$ be a dense and clonal family of $r$-graphs, $n \in \N$, and $1<p<\infty$. Further, suppose that $G\in \cF$ is a graph such that $\rho_p(G) = \Lambda_p(\cF,n)$. Finally, let $\bx \in \pppS$ be a Perron-Frobenius eigenvector of $G$. Then, 
\[ \gamma(G,\bx) = 1 + O_{p,r}(n^{-1}).\] 
\end{proposition}

Two related properties of hypergraph families that appear in the literature (see, e.g., \cite{nikiforovanalyticpreprint}) are ``hereditary'' and ``multiplicative''.

\begin{definition}
A family of graphs $\cP$ is called \emph{hereditary} if for every $G \in \cP$ and every $S \subset V(G)$, $G[S] \in \cP$, where $G[S]$ denotes the induced subgraph $(S,E(G) \cap \binom{S}{2})$.
\end{definition}

\begin{definition}
A family of $r$-uniform graphs $\cP$ is called \emph{multiplicative} if for every $G \in \cP$ and every positive vector $\mathbf{t} \in \Z_+^{V(G)}$, the ``blow-up'' graph $G(\mathbf{t}) \in \cP$, where \[ V(G(\mathbf{t})) = \{(v,j) : v \in V(G), \, j \in [t_v]\}\] and \[ E(G(\mathbf{t})) = \bigg\{\{(v_1,j_1),\ldots,(v_r,j_r)\} \in \binom{V(G(\mathbf{t}))}{r} : \{v_1,\ldots,v_r\} \in E(G)\bigg\}.\]
\end{definition}

\begin{theorem}

If $\cP$ is hereditary and multiplicative, then $\cP$ is clonal.

\end{theorem}

\begin{proof} Let $G \in \cP$, and arbitrarily choose any two vertices $u,v \in V(G)$. We show that $G_{v \to u} \in \cP$. Without loss of generality, let us assume an ordering for the vertices of $G$: $(u, v, v_3, v_4, \ldots, v_n)$. Let $\bt = (2,1,1,\cdots,1) \in \Z_+^n$ and $H = G(\bt)$.  Let $V(H) = \{u_1, u_2, v, v_3, v_4, \ldots, v_n\}$, where $v, v_3, v_4, \ldots, v_n$ in $H$ are obtained from the vertices with the same names in $G$, and $u_1, u_2$ are obtained by doubling the vertex $u$ of $G$.

Clearly $H \in \cP$, since $\cP$ is multiplicative.
On the other hand, $H - v \in \cP$, since $\cP$ is hereditary, where $H - v := H[V(H) - v]$.

Now consider the map $V(H - v) \to V(G_{v \to u})$ where $v_i \mapsto v_i$ for $3 \le i \le n$, $u_1 \mapsto u$, and $u_2 \mapsto v$. This extends to an isomorphism from $H - v$ to $G_{v \to u}$, and hence $G_{v \to u} \in \cP$ as desired.
\end{proof}

Next, for the three properties presented above, we give a collection of counterexamples demonstrating that we have essentially no other implications between them. For this, we shall use a few notions of closure for families of hypergraphs.  In each case, given a property $P$ that a family of graphs may possess, we implicitly define the closure of the family $\cF$ with respect to $P$ as the intersection of all families $\cF' \supseteq \cF$ satisfying $P$.  

\begin{theorem}
For each of the following, there exist examples for which the following implications do not hold.
\begin{enumerate}
    \item $\cP$ clonal $\nRightarrow \cP$ is hereditary or multiplicative.
    \item $\cP$ is hereditary $\nRightarrow \cP$ is multiplicative.
    \item $\cP$ is multiplicative $\nRightarrow \cP$ is hereditary.
    \item $\cP$ is clonal and hereditary $\nRightarrow \cP$ is multiplicative.
    \item $\cP$ is clonal and hereditary and infinite $\nRightarrow \cP$ is multiplicative.
    \item $\cP$ is clonal and multiplicative $\nRightarrow \cP$ is hereditary.
    \item $\cP$ is multiplicative $\nRightarrow \cP$ is clonal.
    \item $\cP$ is hereditary $\nRightarrow \cP$ is clonal.
\end{enumerate}
\end{theorem}

%%[Include the other relationships here] <<- WHAT DOES THIS MEAN?
\begin{proof} \phantom{}
\begin{enumerate}
    \item Consider any infinite clonal family $\cP$, and remove all graphs with an even number of vertices to obtain a subfamily $\cP'$.  Since cloning does not change the same number of vertices, $\cP'$ is still clonal.  However, this family is neither hereditary nor multiplicative, as those properties each imply the presence of at least one graph with an even number of vertices.
    \item We present an example given in \cite{nikiforovanalyticpreprint}.  Consider the family $\cP = \fF(C_4)$ of all $2$-graphs with no induced $C_4$.  This is clearly a hereditary family, but because $C_4 = K_2(2,2) \not \in \cP$ while $K_2 \in \cP$, it is not multiplicative.
    \item Consider any graph $G$ on more than $1$ vertex, and define $\cP$ to be every possible blow-up of $G$.  Then $\cP$ is clearly multiplicative, but because $K_1 \not \in \cP$, it is not hereditary.
    \item Consider the clonal and hereditary closure of any finite graph. Because cloning vertices and taking induced subgraphs never increases the number of vertices, the result is always a finite, clonal, and hereditary family.  However, multiplicative families cannot be finite. 
    \item Take $\cP_1$ to be the family of $3$-graphs with chromatic number bounded by $10$, and let $\cP_2$ be the clonal and hereditary closure of $K_{100}$. Then $\cP_1$ is clonal and hereditary -- neither operation can increase the chromatic number -- and hence so is $\cP =\cP_1 \cup \cP_2$. However, note that $K_{100}$ is maximal in $\cP$, since if $G \supseteq K_{100}$ and $G \in \cP$, then since $\chi(G)>10$, it follows that $G \in \cP_2$ which implies $G \subseteq K_{100}$, and hence $G = K_{100}$. However, a multiplicative family cannot have a maximal element. 

%    \textcolor{blue}{May be true: any infinite clonal, hereditary family where at least one of the members has an edge - must contain all complete multipartite graphs. This is because one can repeatedly clone one of the vertices of the edge to get a $K_{1,n-1}$ centered on the other vertex, and then from there we can clone the center or leaves to get other complete multipartite graphs. Have not thought of r-graphs yet.}
    \item Consider any clonal and multiplicative family $\cP$ -- all $3$-colorable graphs, for example.  Now, let $\cP'$ contain all elements of $\cP$ with at least $100$ vertices. Then $\cP'$ is still clonal and multiplicative, but it is clearly not hereditary.
    \item Consider as $\cP$ the multiplicative closure of $K_{100}$.  All elements of $\cP$ have chromatic number $100$, but cloning any single vertex of $K_{100}$ onto any other vertex results in a graph with chromatic number $99$.  Therefore, $\cP$ is not clonal.
    \item Let $\cP$ be the collection of all finite complete graphs. This is hereditary, since an induced subgraph of a complete graph is complete. On the other hand, any cloning operation on a complete graph results in a non-complete graph.
\end{enumerate}
\end{proof}

%\begin{remark} \label{rem: clonallitexample} In \cite{xuhuwang}, the authors consider family of blow-ups of $\tetra{k}{r}$. This family is clonal, even though it is not of the form $\fF(\cH)$ for any $\cH$, and hence our results apply to it.
%\end{remark}
%[Please check this is actually true -- I worry that if you don't blow-up one of the vertices, then cloning the resulting graph might be a problem]
%%% JC: Yeah, I agree this is a problem.  We could just reference one of the examples for clonal =/=> hereditary above, or even just say "such things exists by Theorem 4.3".  However, the examples there are not very compelling for this particular situation, nor do they connect with the current reference mentioned.  Another option: all blow-ups of K^{(r)}_{2,1,\ldots,1} with fixed k >= r-1 many 1's. Why are we referring to that [42] paper here, btw? 
%\textcolor{blue}{We can take the clonal closure of the family of \{all blow-ups of $K_k^{(r)}$\} to get a clonal family. Is this the family of all blow-ups of one of the $\{K_l^{(r)}, l \le k\}$?}

\section{Upper bound on the principal ratio} \label{sec: principalupper}
Theorem~\ref{thm: principalupper} will follow from the following Proposition~\ref{prop: principalupper}, and so we prove that first.

\begin{proof}[Proof of Proposition~\ref{prop: principalupper}]
Since $\gamma(G,\bx) \geqs 1$ trivially, it suffices to show that $\gamma(G,\bx) \leqs 1 + O_{p,r}(n^{-1})$. To do this, we will first prove the cheaper bound $\gamma \ll_{p,r} 1$, and then bootstrap the result.

Let $z,u \in V(G)$ be vertices so that
\begin{equation} \label{eqn: gammadef} \gamma(G,\bx) = \frac{x_z}{x_u}. \end{equation}
In other words, $x_z = x_{\max}$ and $x_u = x_{\min}$. 
Applying Lemma~\ref{lem: eigen} to $z$, we see that
\[ \rho_p(G) x_z^{p-1} = \frac{1}{r} \frac{\partial}{\partial x_z} L_G(\bx) = (r-1)!\sum_{\substack{e \in E(G)\\z \in e}} x^{e - z}. \]
Separating the terms with $u \in e$, 

\[ \frac{\rho_p(G)}{(r-1)!} x_z^{p-1} = x_u \sum_{\substack{e \in E(G)\\\{u,z\} \subseteq e}}x^{e - \{u,z\}} + \sum_{\substack{e \in E(G)\\z \in e, u \notin e}} x^{e - z}. \]
Now, setting
\begin{equation} \label{eqn: beta} \beta := \sum_{\substack{e \in E(G)\\\{u,z\} \subseteq e}}x^{e - \{u,z\}} \leqs \frac{n^{r-2}}{(r-2)!}x_z^{r-2} , \end{equation}
we can rewrite the above as
\begin{equation} \label{eqn: zeigeneqn} \sum_{\substack{e \in E(G)\\z \in e, u \notin e}} x^{e - z} = \frac{\rho_p(G)}{(r-1)!} x_z^{p-1} - \beta x_u. \end{equation}
On the other hand, applying Lemma~\ref{lem: eigen} to $u$, 
\begin{equation} \label{eqn: ueigeneqn} \sum_{\substack{e \in E(G)\\ u \in e}} x^{e - u} = \frac{\rho_p(G)}{(r-1)!} x_u^{p-1}.
\end{equation}
Then, let $G' = G_{u \to z}$. By an easy calculation,
\[ L_{G'}(\bx) = L_G(\bx) - r! x_u  \sum_{\substack{e \in E(G)\\u \in e}} x^{e - u} + r! x_u \sum_{\substack{e\in E(G)\\z\in e,u\notin e}} x^{e-z}. \]
Now, $L_{G'}(\bx) \leqs \rho_p(G') \leqs \Lambda_p(\cP,n)$. Furthermore, $L_G(\bx) = \Lambda_p(\cP,n)$ by assumption. Combining these with the above, we see that
\[ \Lambda_p(\cP,n) \geqs \Lambda_p(\cP,n) -  r! x_u \sum_{\substack{e \in E(G)\\u \in e}} x^{e - u} + r! x_u \sum_{\substack{e\in E(G)\\z\in e,u\notin e}} x^{e-z}. \]
Now, substituting \eqref{eqn: zeigeneqn} and \eqref{eqn: ueigeneqn} above, we get
\[ 0 \geqs -r! x_u\bigg(\frac{\rho}{(r-1)!} x_u^{p-1}\bigg) + r! x_u \bigg( \frac{\rho}{(r-1)!} x_z^{p-1} - \beta x_u\bigg), \]
where $\rho = \rho_p(G)$. Dividing throughout by $r\rho x_u^p$, letting $\beta' := (r-1)! \beta x_u^{2-p} \rho^{-1}$, and recalling that $\gamma = \gamma(G,\bx) = x_z/x_u$, we find that
\[ 0 \geqs -1 + \gamma^{p-1} - \beta', \]
and hence,
\begin{equation} \label{eqn: g ll b'} \gamma \leqs (1 + \beta')^{1/(p-1)}. \end{equation}

Since $\bx \in \ppS$, we can drop all terms from the left of \eqref{eqn: ueigeneqn} which do not involve $z$ to get that
\[ \beta x_z = \sum_{\substack{e \in E(G)\\ \{u,z\} \in e}} x^{e - u} \leqs \frac{\rho_p(G)}{(r-1)!} x_u^{p-1}.\]
Ignoring constants that depend on $r$ and dividing throughout by $\rho x_z x_u^{p-2}$, this says that
\[ \beta' \ll_{p,r} \gamma^{-1}. \]
Now, if $\beta' \leqs 1$, then by \eqref{eqn: g ll b'}, $\gamma$ is bounded. If not, \eqref{eqn: g ll b'} combined with the above gives us that
\[ \gamma \ll_{p,r} (\beta')^{1/(p-1)} \ll_{p,r} \gamma^{-1/(p-1)}, \]
and hence $\gamma \ll_{p,r} 1$ as desired. In particular, this tells us that $x_u \leqs x_z \ll_{p,r} x_u$. Since $\pnorm{\bx} = 1$, this implies that
\begin{equation} \label{eqn: cheap-flatness} \frac{1}{n^{1/p}} \ll_{p,r} x_u \leqs x_z \ll_{p,r} \frac{1}{n^{1/p}}. \end{equation}

Returning now to \eqref{eqn: g ll b'}, we see that the proposition follows if $\beta' \ll_{p,r} n^{-1}$. Since $\cP$ is dense, Lemma~\ref{lem: density} tells us that $\rho \gg_{p,r} n^{r(1 - 1/p)}$. Thus, by \eqref{eqn: beta} and \eqref{eqn: cheap-flatness},
\begin{equation*} \begin{split} \beta' 
&= \frac{(r-1)! \beta}{\rho x_u^{p-2}}  \\
& \leqs \frac{(r-1)n^{r-2} x_z^{r-2}}{\rho x_u^{p-2}} \\
& \ll_{p,r} \frac{n^{r-2} (n^{-1/p})^{r-2}}{n^{r(1-1/p)} (n^{-1/p})^{p-2}} = n^{-1}, \end{split} \end{equation*}
completing the proof.
\end{proof}

To deduce Theorem~\ref{thm: principalupper} from the above, it suffices to note that, as per Remark~\ref{rem: erdosdensity} and Proposition~\ref{prop: 2-cover give clonal}, if $\cH$ satisfies the hypotheses of the theorem, then $\fF(\cH)$ is a dense clonal family of $r$-graphs, and hence one can apply the above proposition with $\cF = \fF(\cH)$. 

\section{Proof of Theorem~\ref{thm: CTforpnorm}} \label{sec: CTforpnorm}
We suppress the notational dependence on $G$ and $\bx$, as they are fixed. Note that $\gamma = \infty$ if $\bx \in \ppS \setminus \pppS$, so there is nothing to show. Thus, we can assume that $\bx \in \pppS$. Letting $z,u \in V(G)$ be as in the proof of Proposition~\ref{prop: principalupper}, viz. $x_z = x_{\max}$ and $x_u = x_{\min}$, we see that $x_u > 0$ and that \eqref{eqn: gammadef} holds.

Applying Lemma~\ref{lem: eigen} to the vertex $v$ satisfying $\deg v = \Delta$, we find that
\[ \rho_p x_{z}^{p-1} \geqs \rho_p x_v^{p-1} = (r-1)! \sum_{\substack{e \in E(G)\\ v \in e}} x^{e-v} \geqs  (r-1)! \Delta x_{u}^{r-1}, \]
while applying it to the vertex $w$ satisfying $\deg w = \delta$, we get
\[ \rho_p x_{u}^{p-1} \leqs \rho_p x_w^{p-1} = (r-1)! \sum_{\substack{e \in E(G)\\ w \in e}} x^{e-w} \leqs  (r-1)!\delta x_{z}^{r-1}. \]
Multiplying the end-points of the previous two inequalities, we obtain
\[  (r-1)! \rho_p\delta x_{z}^{p+r-2} \geqs  (r-1)!\rho_p\Delta x_{u}^{p+r-2},\]
whence the theorem follows by recalling \eqref{eqn: gammadef} and some rearrangement. 

\section{Tight families and progress towards Conjecture~\ref{conj: bridgelessifftight}} \label{sec: tightfamilies}
We begin our presentation of the evidence for Conjecture~\ref{conj: bridgelessifftight} by proving one direction, viz. Theorem~\ref{thm: bridgelessimpliestight}. 

\begin{proof}[Proof of Theorem~\ref{thm: bridgelessimpliestight}] For the sake of contradiction, suppose that $H$ is a $k$-bridgeless $r$-graph for which $\fF(H)$ is not $k$-tight. It follows that there exists a graph $G$ which is edge-maximal with respect to being $H$-free but is not $k$-tight. Thus, for any $e \notin E(G)$, one has that $G+e$ has a copy of $H$. Since $G$ is not $k$-tight, there is a set $U \subsetneq V(G)$ which induces an edge and so that no $e \in E(G)$ satisfies 
\begin{equation} \label{eqn: ecapu}  k \leqs | e \cap U | \leqs r -1. \end{equation}
Since $U$ induces an edge, $|U| \geqs r$. Let $C \subseteq U$ be an arbitrary subset satisfying $|C| = r-1$. Further, since $U$ is proper, there is an element $b \in V(G) \setminus U$. Define $f_1 = C \cup \{b\}$. Clearly, since $|f \cap U| = r-1 \geqs k$, it follows that $f_1 \notin E(G)$. Since $G$ is $H$-free but $G+f_1$ is not, it follows that $G+f_1$ contains a copy $H'$ of $H$ such that $f_1 \in E(H')$. 

By construction, if we set $A = V(H') \cap U$ and $B = V(H') \setminus U$, then $|f_1 \cap A| \geqs r-1 \geqs k$ and $f_1 \cap B \neq \emptyset$. 
However, since $H'$ is $k$-bridgeless, there must then be another edge $f_2 \in E(H')$ with the property that 
\[ r-1 \geqs |f_2 \cap U| \geqs |f_2 \cap A| \geqs k. \]
But then, since $f_2 \neq f_1$ and $f_2 \in E(G+f_1)$, $f_2$ must actually be an edge in $G$, contradicting the fact that $G$ has no edges satisfying \eqref{eqn: ecapu}. This completes the proof of Theorem~\ref{thm: bridgelessimpliestight}.
\end{proof}

In the converse direction of Conjecture~\ref{conj: bridgelessifftight}, we prove only a partial result. To state this result, we need a couple of definitions.

\begin{definition}
Let $H$ be a connected $r$-graph and $\lambda \vdash r$ be a nontrivial partition. An edge $e$ in $H$ is called a \emph{$\lambda$-plateau} if 

\[ H - e = \bigsqcup_{j=1}^\ell H_j \]
with the property that $|e \cap V(H_j)| = \lambda_j$ for $1 \leqs j \leqs \ell$. In particular, this implies that $H - e$ is disconnected.
\end{definition}

\begin{remark} Note that we are not mandating that the induced subgraph on $V_j$ is connected. Thus, one could have that $e$ is a $\mu$-plateau also, where $\mu \vdash r$ is a partition that refines $\lambda$.
\end{remark}

\begin{remark}
%Observe that a $\lambda$-plateau is a strict $k$-bridge for all $k \geqs \lambda_\ell$. Furthermore, any strict $k$-bridge must be a $\lambda$-plateau for some $\lambda \vdash r$ having $k$ as a part. Finally, a $1$-bridge is a $\lambda$-plateau for some $\lambda \vdash r$ with $\lambda_\ell = 1$ (see also Remark~\ref{rem: 1bridge}).
Observe that a $1$-bridge is a $\lambda$-plateau for some $\lambda \vdash r$ with $\lambda_\ell = 1$ (see also Remark~\ref{rem: 1bridge}).
\end{remark}
\begin{definition} \label{def: plateaud}
Let $1 \leqs k \leqs r-1$ and $H$ be an $r$-graph. Then, $H$ is called \emph{$k$-plateaued} if for every $\lambda \vdash r$ with $k \leqs \lambda_1 \leqs r-1$, we have that $H$ contains a $\lambda$-plateau. 
\end{definition}
\begin{remark} It is easy to create examples of $k$-plateaued graphs for every $1 \leqs k \leqs r-1$ by a greedy construction.  
\end{remark}

The result is then:
\begin{theorem} \label{thm2: plateau}
Let $1 \leqs k \leqs r-1$, and let $H$ be a $k$-plateaued $r$-graph on at least $r+1$ vertices. Then, $\fF(H)$ is not $k$-tight.
\end{theorem}

While it follows immediately from Theorem~\ref{thm2: plateau}, it is instructive to prove Corollary~\ref{corr: 1bridgelessiff1tight} first, as a key idea will be reused. %For this, we need the following lemma.

\begin{proof}[Proof of Corollary~\ref{corr: 1bridgelessiff1tight}]
If $H$ is a $1$-bridged $2$-graph, then $H - e = H_1 \sqcup H_2$ where $e$ is an edge with one vertex $u_1$ in $H_1$ and one vertex $u_2$ in $H_2$. Let $t = |V(H)| - 1$, and define $G = \ell K_t$ to be $\ell$ disjoint copies of the complete graph on $t$ vertices. %Further, let $U$ be the vertex set of any of these complete graphs. We claim that $(H,G,U)$ satisfy the hypotheses of Lemma~\ref{lem: keyconstructionlemma}, completing this case. 

We claim that $G \in \fF(H)$ is edge-maximal. To see that $G \in \fF(H)$, recall that $H$ is connected. Thus, if $G$ contains a copy of $H$, then it must be contained in a component of $G$. On the other hand, $t < |V(H)|$ whence the components of $G$ do not have enough space for $H$, giving a contradiction. To see the maximality of $G$, note that $G+f$ for any $f \notin V(G)$ must satisfy that both vertices of $f$ are in different cliques. However, this means that a copy of $H$ exists in $G+f$ with $f$ as the $1$-bridge, since $|V(H_j)| \leqs |V(H)| - 1 = t$, and copies of $H_j$ can be found in the cliques with $f$ corresponding to the vertices labelled $\{u_1,u_2\}$.
\end{proof}

\begin{figure}[h]
\includegraphics[width=5in]{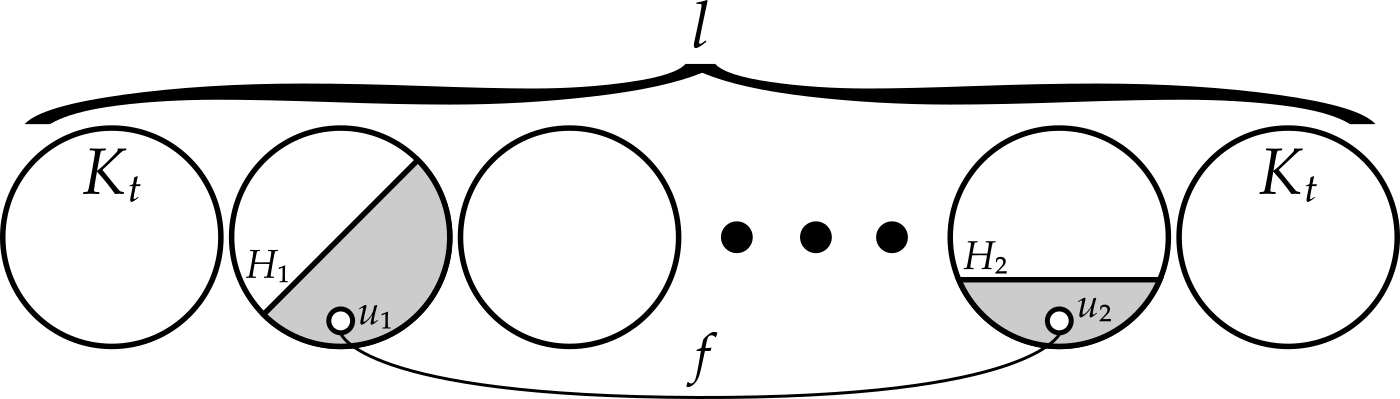}
\centering
\caption{Illustration of the argument for Corollary \ref{corr: 1bridgelessiff1tight}.}
\end{figure}

Embedded in the above proof is an easy lemma.

\begin{lemma}
Let $H$ be an $r$-graph with a $\lambda$-plateau from some nontrivial $\lambda \vdash r$ with $\ell$ parts. Further, let $L_{\ell,\lambda}$ be the $r$-graph consisting of $\ell$ disjoint copies of $\tetra{t}{r}$, together with an edge $f$ such that $f$ intersects the $j$th clique on exactly $\lambda_j$ vertices. Then, $L_{\ell,\lambda}$ contains a copy of $H$.
\end{lemma}

\begin{proof}[Proof of Theorem~\ref{thm2: plateau}]

Let $t = |V(H)|$ be the number of vertices of $H$. Define $G'$ to be a collection of $\ell$ disjoint $K_{t-1}$. Trivially, $G'$ is $H$-free as any copy of $H$ must be contained in some component of $G'$, but no component of $G'$ has enough vertices. Take $G$ to be an $H$-saturated graph on $V(G')$ containing $G'$ as a subgraph.  
  
We claim that $G$ will not be $k$-tight. To witness this, let $J_1, \cdots, J_\ell$ be the vertex sets of the cliques contained in $G'$, and set $A = J_1$ and $B = G \setminus J_1$. It is trivial that $A$ induces an edge in $G$, since $t-1 \geq r$. Suppose there were an edge $e$ with $k \leqslant |e \cap A| \leqslant r-1$. Suppose that $e$ has $\mu_j$ vertices in $J_j$. Then some of the $\mu_j$ may be zero, and the ones which are non-zero would form a partition $\lambda$ (after some rearrangement of indices) of $r$. Furthermore, $\lambda_1 \geqs k$, since the part of $\lambda$ corresponding to $\mu_1$ is clearly $\geqs k$. However, this creates a copy of $L_{\lambda,n}$, and hence by the previous lemma, a copy of $H$. But since $G$ is edge-maximal in $\fF(H)$ this cannot be the case, and hence by contradiction, there is no such edge $e$. This completes the proof.  
\end{proof}

%The second property on $(H,G,U)$ is trivial since $U$ induces an edge, but contains no outgoing edges. 

%The third property 

\bibliography{hypergraphspex}
\bibliographystyle{plain}

\end{document}